\documentclass{amsart}
\usepackage{amsfonts}
\usepackage{amsmath}
\usepackage{amssymb}

\newtheorem{theorem}{Theorem}[section]

\newtheorem{lemma}[theorem]{Lemma}

\newtheorem{proposition}[theorem]{Proposition}

\newtheorem{corollary}[theorem]{Corollary}

\theoremstyle{definition}

\newtheorem{example}[theorem]{Example}

\newtheorem *{Theorem A}{Theorem A}
\newtheorem *{Fraenkel's Conjecture}{Fraenkel's Conjecture}
\newtheorem *{The Three Gap Theorem}{The Three Gap Theorem}

\begin{document}
\title{A different approach to the Fraenkel Conjecture for low $n$ values}

\author{Ofir Schnabel}
\address{Department of Mathematics, University of Haifa, Haifa 3498838, Israel}
\email{os2519@yahoo.com}

\author{Jamie Simpson}
\address{Department of Mathematics and Statistics, Curtin University, Perth 6102, Western Australia}
\email{simpson@maths.curtin.edu.au}

\begin{abstract}
We present a new approach to deal with Fraenkel's conjecture, which
describes how the integers can be partitioned into sets of
rational Beatty sequences, in the case where the numerators of the
moduli are equal. We use this approach to give a new proof of the known $n=4$
case when the numerators are equal.
\end{abstract}
\date{\today\vspace{-0.5cm}}
\maketitle
\bibliographystyle{abbrv}

\section{Introduction}\label{intro}\pagenumbering{arabic} \setcounter{page}{1}
A set of arithmetic progressions which partitions the integers is
called a \textit{disjoint covering system} (DCS). A classic result
from the 1950's (see \cite{EG}) shows that in any DCS there must
be two arithmetic progressions with the same common modulus, that
is, any DCS admits \textit{multiplicity}. Since then there has
been considerable study of the ways the integers can be
partitioned into arithmetic progressions and there have  been
generalizations of this concept. One of these generalizations is a
disjoint covering system of Beatty sequences. A \textit{Beatty
sequence}  is a sequence $S(\alpha ,\beta)= \{\lfloor \alpha
n+\beta \rfloor :n\in \mathbb{Z}\}$, where $\alpha ,\beta \in
\mathbb{R}$ ($\lfloor \alpha n+\beta \rfloor$ is the  integer part
of $\alpha n+\beta$). Here, $\alpha$ is called the
\textit{modulus} of the sequence $S(\alpha ,\beta)$. Similarly to
partitioning the integers into arithmetic progressions, a
\textit{disjoint covering system of Beatty sequences}  is a set of
Beatty sequences $\{S(\alpha_i ,\beta_i)\}_{i=1}^k$ such that
every integer belongs to  exactly one Beatty sequence. Clearly, if
$\alpha_i\in \mathbb{Z}$ for all $i$ then the system
 $\{S(\alpha_i ,\beta_i)\}_{i=1}^k$ is a DCS.
By density arguments, for a system $\{ S(\alpha _i,\beta
_i)\}_{i=1}^n$,
\begin{equation}\label{eq:weight}
\sum_{i=1}^n \frac{1}{\alpha _i}=1.
\end{equation}
The concept of a covering system of Beatty sequences is attributed to Samuel Beatty \cite{beatty}
who proved in 1926 that if $x,y$ are irrational positive numbers
such that
\begin{equation}\label{eq:Beatty}
\frac{1}{x}+\frac{1}{y}=1 ,
\end{equation}
then the sequences $\{ix\}_{i=1}^\infty , \{iy\}_{i=1}^\infty $
contain one and only one number between each pair of consecutive
natural numbers. By~\eqref{eq:weight} any irrational numbers $x,y$
satisfy $\frac{1}{x}+\frac{1}{y}=1$ if and only if the
$S(x,0),S(y,0)$ induce a partition on the natural numbers. In
fact, with appropriate $\beta _1,\beta _2$  $S(x,\beta
_1),S(y,\beta _2)$ partition the whole set of integers. Apparently
John William Strutt (Lord Rayleigh) in his book, ``The theory of
sound" from 1877 \cite{strutt} was the first to refer, indirectly,
to such systems.

As with DCS, we say that a disjoint covering  system of Beatty
sequences $\{S(\alpha_i ,\beta_i)\}_{i=1}^k$ admits
\textit{multiplicity} if there exist $1\leq i<j\leq k$ such that
$\alpha_i=\alpha_j$. It is natural to ask if, as is the case with
DCS, any such system of Beatty sequences has multiplicity. It
follows from Beatty's result that the answer is negative if the
integers are partitioned by only two Beatty sequences. However, if
$k>2$, Graham \cite{graham73} showed that any system
$\{S(\alpha_i ,\beta_i)\}_{i=1}^k$ with irrational moduli admits
multiplicity.

We are left with the case where all the moduli are rational, in which case we call the sequences \emph{rational Beatty sequences}. We will call a set of rational Beatty sequences which partition the integers a \emph{disjoint covering system of rational Beatty sequences}, or a DCS of RBS or, when there is no danger of confusion, just a DCS.
It is easy to see that there are ways of partitioning the integers with rational Beatty
systems without multiplicity. For example, a DCS with 3
sequences is: $S(7/4,0), S(7/2,-1),
S(7,-3)$. Fraenkel \cite{fraenkel73} showed that for
every positive integer the system
\begin{equation}\label{eq:rational}
\{S(2^n-1/2^{n-i}, -2^{i-1}+1)\}_{i=1}^n
\end{equation}
is a DCS with distinct moduli. Fraenkel conjectured
that the systems~\eqref{eq:rational} are essentially the only ones
(up to translation) without multiplicity. More precisely,
\begin{Fraenkel's Conjecture}\label{con:fraenkel} (Fraenkel, see \cite{simpson2004})\\
If  $\{ S(\alpha _i,\beta _i)\}_{i=1}^n$ form a DCS, with
$n>2$ and $\alpha _1<\alpha _2<\cdots <\alpha _n$ then
$\alpha _i=p/q_i$, where $p=2^n-1, q_i=2^{n-i}$, for $i=1,2,\ldots, n$.\\
\end{Fraenkel's Conjecture}
Fraenkel's conjecture has been proved for $n\leq 7$. For $n=3$ by
Morikawa \cite{morikawa}, for $n=4$ by  Altman, Gaujal and
Hordijk \cite{altman}, for $n=5$, and $6$  by Tijdeman \cite{tijdeman}
and for $n=7$ by Bar\'{a}t and Varj\'{u} \cite{barat}.

In this note we present a new approach to deal with Fraenkel's
conjecture which will hopefully be easier to generalize to larger
values of $n$ under the assumption that all the numerators of the
moduli are equal. This assumption will enable us to associate with
each $q_i$ a set $B_i$ containing $q_i$ $p$-th roots of unity.
Moreover, for $i\neq j$ the sets $B_i$ and $B_j$ are disjoint and
therefore the sets $B_i,$ $1\leq i\leq n$ partition the group of
all $p$-th roots of unity. This will allow us to carry out some
manipulations which solve the problem for $n=3,4$, and which might
be adapted to prove, under the equal numerators assumption, the
conjecture for larger values of $n$ (see e.g. Lemma~\ref{Ofir}).
We record the $n=4$ case as
\begin{Theorem A}\label{th:BS2}(see \cite{altman})\\
Let $\{ S(p/q_i,\beta _i)\}_{i=1}^4$ be a disjoint covering system of rational Beatty sequences. If the
moduli are distinct then $p=15$ and $\{q_1,q_2,q_3,q_4\}=\{8,4,2,1\}$.
\end{Theorem A}
We mention that a proof for the $n=3$ case can be easily deduced
from the proof of the $n=4$ case.

\section{From Beatty sequences to a partition of $\{0,1,\ldots ,p-1\}$}

We start by introducing a relation between disjoint covering
systems of Beatty sequences $\{ S(p/q_i,\beta _i)\}_{i=1}^n$ and
partitions of the set $\{0,1,\ldots ,p-1\}$ which will be important
in the proof of Theorem A.

It is shown in \cite{simpson2004} that for disjoint covering
system of Beatty sequences $\{ S(\frac{p_i}{q_i},\beta
_i)\}_{i=1}^n$, there exists a polynomial $f(z)$ whose
coefficients all equal to 1 such that
\begin{equation}\label{eq:generating}
f(z)+\sum_{i=1}^n \frac{z^{b_i}}{1-z^{p_i}} \cdot
\frac{1-z^{\overline{q}_iq_i}}{1-z^{\overline{q}_i}}
=\frac{1}{1-z},
\end{equation}
 where $\overline{q}_i$ is the smallest non-negative integer satisfying
 \begin{equation}\label{eq:simpsonqi}
{\overline{q}_iq_i} \equiv -1\pmod{p_i}.
\end{equation}
However, in our situation, $p_i=p_j=p$ for every $1\leq i,j\leq
n$. Therefore, if $\xi$ is a $p$-th primitive root of unity, by
taking the limit as $z$ goes to $ \xi$
in~\eqref{eq:generating} we get
\begin{equation}\label{eq:vanishing3}
\sum_{i=1}^n \frac{\xi ^{b_i}}{1-\xi ^{\overline{q}_i}}\cdot
(1-\xi^ {q_i\overline{q}_i})=0.
\end{equation}
Consequently, by~\eqref{eq:simpsonqi}
\begin{equation}\label{eq:vanishing}
\sum_{i=1}^n \frac{\xi ^{b_i}}{1-\xi ^{\overline{q}_i}}=0,
\end{equation}
and then
\begin{equation}\label{eq:roots}
\sum_{i=1}^n \xi ^{b_i}(1+\xi ^{\overline{q}_i}+\xi
^{2\overline{q}_i}+\cdots +\xi ^{(q_i-1)\overline{q}_i})=0.
\end{equation}
Equation~\eqref{eq:roots} induces a disjoint cover of the
\textit{$p$-th} roots
 of unity with corresponding sets $M_i$ for $1\leq i\leq n$ where
\begin{equation}
M_i=\{\xi ^{b_i+j\overline{q}_i}\}_{j=0}^{q_i-1}.
\end{equation}
Multiplying the sets by $\xi ^{-b_1}$ we may assume
  $M_1=\{\xi ^{j\overline{q}_1}\}_{j=0}^{q_1-1}$.\\
 Let $\gamma:\xi ^i\longmapsto \xi ^{-q_1\cdot i}$ be an automorphism
 of the group of roots of unity $\{1,\xi,\xi ^2,\ldots \xi ^{p-1}\}$.
 This is indeed an automorphism since $(q_1,p)=1$.
Let $\tilde{q_i}$ be the smallest non-negative integer satisfying
$\tilde{q_i}\equiv -q_1\cdot \overline{q}_i\pmod{p}$, $(1\leq i\leq
n)$.
 Then $\tilde{q_1}=1$ and since $\{\overline{q}_i\}_{i=1}^n$ are
  distinct then so are $\{\tilde{q_i}\}_{i=1}^n$. We get
 a new cover of the \textit{$p$-th} roots
 of unity with $\tilde{q_1}=1$ which in turn induces a partition
 $$\{0,1,\ldots ,p-1\}=\mathop{\dot{\bigcup}}_i B_i,$$
 where
 \begin{equation}\label{eq:BiTG}
B_i=\left\{\tilde{b_i}+j\tilde{q_i}
\pmod{p}\right\}_{j=0}^{q_i-1}.
\end{equation}
Here, $\tilde{b_i}$ is the smallest non-negative integer satisfying
$\tilde{b_i}\equiv -q_1\cdot b_i\pmod{p}$ for $2\leq i\leq n$.  This implies that  $\tilde{b_1}=0$ and $\tilde{q_1}=1$ so that $B_1=\{0,1,\dots,q-1\}.$

From now on we will use the correspondence between disjoint
covering system of Beatty sequences $\{ S(p/q_i,\beta
_i)\}_{i=1}^n$ and the sets $B_i$ without saying explicitly what
that correspondence is.
\section{Beatty sequences and TG-sequences}
In this section we provide some preliminary results concerning
disjoint covering system of Beatty sequences and  introduce
and study TG-sequences. In the next lemma we write $B_1+\tilde{q}_2$ for $\{b+\tilde{q}_2:b \in B_1\}$.
\begin{lemma}\label{Ofir}
Let $S(p/q_1,0)$, $S(p/q_2,b_2)$ be disjoint Beatty sequences and
let
$$B_1=\{0,\dots,q_1-1\} \quad\textit{and }
B_2=\{\tilde{b}_2,\tilde{b}_2+\tilde{q}_2,\dots,\tilde{b}_2+(q_2-1)\tilde{q}_2\}$$
be the corresponding sets defined in (\ref{eq:BiTG}). If
\begin{equation} \label{e0.2}
B_1\cap (B_1+\tilde{q}_2)\not = \emptyset
\end{equation}
 then $B_2$ is a arithmetic progression in $[q_1,p-1]$
with common modulus $\tilde{q}_2$ or $p-\tilde{q}_2$.
\end{lemma}
\begin{proof}
Condition (\ref{e0.2}) implies that either $\tilde{q}_2<q_1$ or
$\tilde{q}_2>p-q_1$. For $i \in \{0,\dots,q_1-2\}$ the minimum of $\tilde{b}_2+(i+1)\tilde{q}_2-(\tilde{b}_2+i\tilde{q}_2) \mod p$ and $\tilde{b}_2+i\tilde{q}_2-(\tilde{b}_2+(i+1)\tilde{q}_2) \mod p$ is the minimum of $\tilde{q}_2 \mod p$ and $p-\tilde{q}_2 \mod p$, which is less than $q_1-1$.  So $B_1$ does not fit in this gap.  Therefore we have
$$ 1<q_1 <\tilde{b}_2<\tilde{b}_2+\tilde{q}_2< \dots <\tilde{b}_2+(q_2-1)\tilde{q}_2<p$$
or
$$1<q_1 <\tilde{b}_2+(q_2-1)\tilde{q}_2< \dots <\tilde{b}_2+\tilde{q}_2<\tilde{b}_2<p.$$
The result follows.
\end{proof}
The following lemma will be useful in the proof of Theorem A.
\begin{lemma} \label{L1} If $S(p_1/q_1,b_1) \cup S(p_2/q_2,b_2)=S(p_3/q_3,b_3)$
then one of the following holds:\\
(a) $\{p_1/q_1,p_2/q_2\}$ is any size 2 subset of
$\{7/1,7/2,7/4\}$,\\
(b) $p_1/q_1=p_2/q_2$,\\
(c) $\{p_1/q_1,p_2/q_2\}=\{p/q,p/(p-2q)\}$ for some $p$ and $q$.
\end{lemma}
\begin{proof} By \cite[Lemma 2]{simpson91}, the complement of $S(p_3/q_3,b_3)$
is $S(p_3/(p_3-q_3),b_3-\overline{q}_3)$. Hence
\begin{equation} \label{e 4.7}
S(p_1/q_1,b_1) \cup S(p_2/q_2,b_2)\cup
S(p_3/(p_3-q_3),b_3-\overline{q}_3)=\mathbb{Z}.
\end{equation} If $p_1/q_1$, $p_2/q_2$ and $p_3/(p_3-q_3)$ are
distinct we have the (proven) $n=3$ case of Fraenkel's conjecture
which gives part (a) of the Lemma. Otherwise two of the moduli
appearing in (\ref{e 4.7}) are equal. If $p_1/q_1 = p_2/q_2$ we
get case (b), and if $p_3/(p_3-q_3)$ equals one of the other
moduli we get case (c).
\end{proof}
We note that in part (c) we have $q<p/2$ so that
$p_3/q_3=p/(p-q)>1/2$.

The sets $B_i$ in~\eqref{eq:BiTG} are all a particular case of
TG-sequences hereby explained. Let $a$ and $d$ be residues modulo
$p$ with $(p,d)=1$, $q$ a positive integer less than $p$ and
consider the set $\{a+id \mod p:i=0,\dots,q-1\}$. Sort this set
into a sequence $a_1\le a_2 \le \dots \le a_q$. Call this sequence
a $TG-sequence$ with $q$ points and modulus $p$. Call the pairs
$(a_i,a_{i+1}), \; i=1,\dots,q-1$ and $(a_{q},a_1)$ the
\emph{gaps} of the sequence. Say the \emph{size} of such a gap is
$a_{i+1}-a_i$ when $i<n-1$ and $p+a_1-a_q$ in the other case.
\begin{The Three Gap Theorem}\label{3 gap} (\cite{3gapth})
The number of distinct gaps sizes in a $TG$-sequence is at most 3,
and if it equals 3 then the largest gap size equals the sum of the
other two. If there is only one gap size then $q=1$.
\end{The Three Gap Theorem}
\begin{proof} The first part of this theorem is well known in a different setting \cite{3gapth}.
Suppose we have only one gap size and it equals $c$. Then $p=cq$
so $c$ divides $p$.  Without loss of generality we assume that
$a$, in the definition of a $TG$-sequence, equals 0. Thus
\begin{equation*}
\{b+ic:i=0,\dots,q-1\} \equiv \{id:i=0,\dots,q-1\} \mod p
\end{equation*}
for some integer $b$.  Therefore, if $q>1$, there exist integers
$i_1$ and $i_2$ such that $0 \equiv b+i_1c \mod p$ and $d \equiv
b+i_2c \mod p$.  Therefore $d \equiv (i_2-i_1)c \mod p$. Since
$(p,d)=1$ we have $(p,c)=1$.  But $c$ divides $p$ so $c=1$ and
$p=q$. This contradicts the definition of a $TG$-sequence and we
conclude $q=1$.
\end{proof}

The reader will appreciate that $TG$ stands for Three Gap. We will
refer to gaps with the smallest size as \emph{small} gaps and the
others as \emph{larger} gaps.

\begin{corollary} \label{c1} Let $B$ be a $TG$-sequence with one larger gap
so that the points in $B$ form an arithmetic progression. Then
using the notation of the definition the common modulus of this
arithmetic progression is either $d$ or $p-d$.
 \end{corollary}
\begin{proof} Using the notation of the definition of a $TG$-sequence, we may assume,
without loss of generality, that $a=0$ so that
$B\equiv\{id:i=0,\dots,q-1\}$ modulo $p$ where $(p,d)=1$. Say this
is the same set as $\{ic+b,i=0,\dots,q-1\}$.  Multiplying each
term by $d^{-1}$ gives
\begin{eqnarray*}
B'&\equiv&\{0,1,\dots,q-1\}\mod p\\
&\equiv&\{d^{-1}b,d^{-1}b+d^{-1}c,\dots,d^{-1}b+(q-1)d^{-1}c\}
\mod p.
\end{eqnarray*} From the right hand side we see that $|B' \cap (B'+
d^{-1}c)|=q-1$, so $B'+ d^{-1}c$ is congruent to $\{1,\dots,q\}$
or $\{p-1,0,\dots,q-2\}$ modulo $p$ which implies that
$d^{-1}c\equiv\pm 1 \mod p$ and so $c\equiv\pm d \mod p$, and the
result follows.
\end{proof}

\begin{corollary} \label{c6} Let $B$ be a $TG$-sequence with two larger
gaps so that the points in $B$ form two arithmetic progressions
with common modulus $c$. Then, using the notation of the
definition, $c$ is congruent modulo $p$ to either $2d$ or $-2d$.
\end{corollary}

\begin{proof} We may assume, without loss of generality, that $a=0$ so that
$B=\{id:i=0,\dots,q-1\}$ where $(p,d)=1$. Let $K_1$ be one
arithmetic progression and $K_2$ the other. Suppose $(q-1)d$ does
not belong to $K_1$ and consider $K_1+d$. This is a subset of $B$.
It cannot intersect both $K_1$ and $K_2$, neither can it be
contained in $K_1$. Therefore $K_1+d \subseteq K_2$.  Similarly
$K_2 \backslash \{(q-1)d\}+d \subseteq K_1$.  So if $0 \in K_1$
then $d \in K_2$, $2d \in K_1$ and so on. Thus
$K_1=\{0,2d,\dots\}$ and $K_2=\{d,3d,\dots\}$, or vice versa. As
in the proof of the last lemma we have $c \equiv \pm2d \mod p.$
\end{proof}
No doubt this can be easily extended to more than 2 larger gaps.
It also follows from the proof that $||K_1|-|K_2||\le 1$.

\begin{corollary} \label{c2}
Let $B$ be a $TG$-sequence with $q_2>1$ points and modulus $p$,
with smallest gap size $c$ and largest $G$, where $G>q_1$ for some
integer $q_1$. Say that it has $k$ larger gaps.

(i) If the sequence has only two gap sizes then
\begin{equation}\label{cor1.5}
p-q_1-q_2 \ge (k-1)(G-1)+(q_2-k)(c-1).
\end{equation}

(ii) If the sequence has three gap sizes then
\begin{equation}\label{cor2.5}
p-q_1-q_2\ge (k-1)(G-c-1)+(q_2-k)(c-1).
\end{equation}
\end{corollary}
The slightly awkward notation here will simplify the applications.
\begin{proof}
(i) Clearly $p$ equals $q_2$ plus the number of points in the
interiors of the gaps.  The $k$ larger gaps each contains $G-1$
points and the $q_2-k$ small gaps each contain $c-1$ points. Thus,
using the assumption $G>q_1$
\begin{eqnarray*}
p&=&q_2+k(G-1)+(q_2-k)(c-1)\\
&\ge& q_1+q_2+(k-1)(G-1)+(c-1)(q_2-k)
\end{eqnarray*}
giving the required result.\\
(ii) We have at least one gap of size $G$ containing $G-1$ points,
$k-1$ other larger gaps of size at least $G-c$ each containing at
least $G-c-1$ points and $q_2-k$ small gaps each containing $c-1$
points. Thus
\begin{equation*}
p-q_2 \ge G-1+(k-1)(G-c-1)+(q_2-k)(c-1).
\end{equation*}
By assumption $G \ge q_1+1$ which establishes the inequality.
\end{proof}

\begin{example}
 Consider the $TG$ sequence
$\{7i:i=0\dots 3 \mod 13\}=\{0,1,7,8\}$.  We have one larger gap
of size 6, another of size 5 and 2 of size 1.  In the notation of
the corollary $p=13$, $q_1=5$, $q_2=4$, $G=6$, $k=2$ and $c=1$.
Then $p-q_1-q_2$ and
$(k-1)(G-c-1)+(q_2-k)(c-1)$ both equal 4.
\end{example}

 \section{Proof of Theorem A}
 By~\eqref{eq:weight} $p=q_1 + q_2 +  q_3+  q_4$. We may also assume that
 \begin{equation} \label{q_1 + q_2 +  q_3+  q_4}
  q_1 \ge q_2 +1 \ge q_3+2 \ge q_4+3 \ge 4
\end{equation}  and $(p,q_i)=1$ for $1\leq i\leq 4$.
 We cannot have equality
throughout (\ref{q_1 + q_2 +  q_3+  q_4}) for then $p=4q_4+6$
which is not relatively prime to each $q_i$. Thus
\begin{equation}
\label{q_1 q_4} q_1\ge q_4+4.
\end{equation}
With the notation of ~\eqref{eq:BiTG}, $B_1=\{0,1,\ldots q_1-1\}$
and $B_2$ is a $TG$-sequence with $q_2$ points and modulus $p$
which contains a gap of size at least $q_1+1$ (as $B_1$ is
disjoint from $B_2$). Say that its small gaps have size $c$ and
that it has $k$ larger gaps. Clearly the largest gap $G\ge q_1+1$ which  is part of the hypothesis of
Corollary \ref{c2}.
\begin{lemma}
With the above notation, $c,k\in \{1,2\}$.
\end{lemma}
\begin{proof}
 Suppose that $B_2$ has only the two gap sizes $G$ and $c$.
Then by (i) of Corollary \ref{c2} and recalling that
$p=q_1+q_2+q_3+q_4$,

\begin{equation*}
q_3+q_4\ge (k-1)(G-1)+(q_2-k)(c-1)\ge (k-1)q_1+(q_2-k)(c-1).
\end{equation*}
If $k\ge 3$ the right hand side is at least $2q_1$, which is
impossible by (\ref{q_1 + q_2 +  q_3+  q_4}). So we assume $k=1$
or $k=2$. If $c \ge 3$ the right hand side is at least
$(k-1)q_1+2(q_2-k)$ which is $2q_2-2$ when $k=1$ and $q_1+2q_2-4$
when $k=2$. In either case we get a contradiction with (\ref{q_1 +
q_2 +  q_3+  q_4}).  Thus if $B_2$ has two gap sizes we have $1\le
c \le 2$ and $1 \le k \le 2$.

Now suppose we have three gap sizes.  These are $G$, $G-c$ and $c$
with $G>G-c>c$ which implies
\begin{equation}
\label{2c+1} G \ge 2c+1
\end{equation} and by (ii) of Corollary \ref{c2}
\begin{equation}
\label{p_3+p_4}q_3+q_4 \ge (k-1)(G-c-1)+(q_2-k)(c-1).
\end{equation}
We consider various combinations of $c$ and $k$ values.

If $c=1$ and $k \ge 3$ (\ref{p_3+p_4}) gives
$$q_3+q_4 \ge (k-1)(q_1-1)\ge 2q_1-2$$ which is impossible by (\ref{q_1 + q_2 +  q_3+
q_4}).

If $c=2$ and $k \ge 3$ (\ref{p_3+p_4}) gives
\begin{eqnarray*}q_3+q_4 &\ge& (k-1)(G-3)+q_2-k\\
&=& (k-1)(G-4)+q_2-1\\
&\ge& 2q_1+q_2-7
\end{eqnarray*} which is again incompatible with (\ref{q_1 + q_2 +  q_3+
q_4}).

If $c\ge 3$ then, using (\ref{2c+1}),
\begin{eqnarray*}q_3+q_4 &\ge& (k-1)(G-c-1)+(q_2-k)(c-1)\\
&=& (k-1)(G-2c)+(c-1)(q_2-1)\\
&\ge& 2q_2-2
\end{eqnarray*} which is again incompatible with (\ref{q_1 + q_2 +
q_3+ q_4}). We have now eliminated all cases, for both two and
three gaps, except those with $1\le c \le 2$ and $1 \le k \le 2$.
\end{proof}
Next by eliminating the cases $(c,k)=(2,2),(1,2)$ and $(1,1)$ we will prove
\begin{proposition}
With the above notation $q_1=2q_2$ and $\tilde{q}_2=2$.
\end{proposition}
\begin{proof}
If $c=2$ and $k=2$ and we have three gaps (the two gap case is
even simpler) we get
\begin{eqnarray*}q_3+q_4 &\ge& G+q_2-5\\
&\ge& q_1+q_2-4
\end{eqnarray*}
so this case is also eliminated leaving the cases $(c,k)=(1,1)$,
$(1,2)$ or $(2,1)$.

Suppose $c=1$ and $k=2$, then by (\ref{cor1.5}) or (\ref{cor2.5})
we have
$$p\ge 2q_1+q_2-1.$$ By Corollary \ref{c6} $2\tilde{q}_2\equiv \pm 1 \mod p$, so that
$2q_1 \equiv \pm q_2 \mod p$. If $2q_1\equiv q_2 \mod p$ then
$p=2q_1-q_2$ which contradicts the above inequality. If instead
$2q_1\equiv -q_2 \mod p$ we have $p=2q_1+q_2$. Since $B_2$
contains 2 larger gaps and $c=1$, each larger gap has size $q_1+1$
and contains $q_1$ points. One gap contains $B_1$ and the other
$B_3 \cup B_4$. Therefore $B_3 \cup B_4$ is a translate of $B_1$. It follows that
$S(p/q_3,b_3)\cup S(p/q_4,b_4)=S(p/q_1,b)$ for some $b$. Lemma
\ref{L1} then says that either $p=7$ which is impossible,
$q_3=q_4$ which is impossible or $q_1>p/2$ which is also
impossible. We conclude that we cannot have  $c=1$ and $k = 2$. We
are left with the cases $(c,k)=(1,1)$ or $(2,1)$.

If $c=1$ and $k=1$ then, by Corollary \ref{c1}, $\tilde{q}_2\equiv \pm 1
\mod p$ so $q_1=q_2$ or $q_1=p-q_2$, both of which are impossible.

We conclude that and $c=2$ and $k=1$  so that $\tilde{q}_2\equiv \pm2
\mod p$. If $\tilde{q}_2=-2$ we have $q_1+2q_2=p$. Then
$B_2=\{q_1,q_1+2,\dots,p-2\}$ and $B_3\cup B_4=
\{q_1+1,q_1+3,\dots,p-1\}$, or vice versa. Either way $B_2$ is a
translate of $B_3\cup B_4$ so that $S(p/q_3,b_3)\cup S(p/q_4,b_4)$
= $S(p/q_2,b)$ for some $b$.  This is impossible by Lemma
\ref{L1}. So $\tilde{q}_2=2$ and $q_1=2q_2$.
\end{proof}

\begin{proposition}
With the notation of Theorem A, $q_2=2q_3$.
\end{proposition}
\begin{proof}
$B_2$ contains $q_2-1$  gaps of size 2 which must be filled by members
of $B_3$ and $B_4$. We thus have
\begin{equation}\label{e2}
q_3+q_4\ge q_2-1.
\end{equation} Suppose $B_1\cap(B_1+\tilde{q}_3)=\emptyset$. Then $\tilde{q}_3 \ge q_1$ so  $B_1+\tilde{q}_3$ can contain
at most one member of $B_3$ and we therefore have $p\ge
2q_1-1+q_3=q_1+2q_2+q_3-1$ which is impossible. Thus $B_1$ and
$B_1+\tilde{q}_3$ have non-empty intersection. By Lemma \ref{Ofir}
$B_3$ is an arithmetic progression with common modulus $c$ equal
to $\tilde{q}_3$ or $p-\tilde{q}_3$. The points in the gaps
between the members of $B_3$ belong to $B_2$ or $B_4$. Since there
are $q_3-1$ such gaps there is at least one gap with at most one
member of $B_4$. Since $\tilde{q}_2=2 \mod p$ this allows at most
2 elements of $B_2$ so $c \le 4$.  If $c=1$ we have $\tilde{q}_1
\equiv \tilde{q}_3$ or $\tilde{q}_1 \equiv -\tilde{q}_3$. The
first implies that $q_1=q_3$ the second that $q_1+q_3=p$, both of
which are impossible. If $c=2$ we have $\tilde{q}_3\equiv \pm
\tilde{q}_2$ which leads to a similar contradiction. If $c=3$ we
can have members of $B_2$ in at most two the gaps in $B_3$. By
considering possibilities as before we find this is impossible. So
$c=4$ which means $\tilde{q}_3 \equiv \pm 4 \mod p$. If
$\tilde{q}_3 \equiv -4 \mod p$ we get $4q_3+q_1 \equiv p \mod p$
so $3q_3 \equiv q_2+q_4 \mod p$. This means $3q_3 =q_2+q_4$. Then
(\ref{e2}) implies that $3q_3 \le q_3+2q_4+1$ which is impossible.
So $\tilde{q}_3 \equiv 4 \mod p$ which implies $q_1=4q_3$ and
$q_2=2q_3$.
\end{proof}
We are now ready to complete the proof of Theorem A.

Substituting in (\ref{e2}) gives $q_4\ge q_3-1$, so, by (\ref{q_1
+ q_2 +  q_3+ q_4}), $q_4=q_3-1$.

By considering $B_1\cap(B_1+\tilde{q}_3)$ and using Lemma
\ref{Ofir} as in the $B_3$ case we find that $B_4$ is an
arithmetic progression, possibly with a single term. Now $B_2$ is
an arithmetic progression with common modulus 2 and $2q_4+2$
terms. The $2q_4+1$ gaps between the terms must be filled with
members of $B_3$ and $B_4$. If $q_4>1$ the common modulus of $B_4$
will be 4 leading to $\tilde{q}_4\equiv \pm \tilde{q}_3 \mod p$,
leading to $q_3=q_4$ or $q_3+q_4=p$, both of which are impossible.
We conclude that $q_4=1$, $q_3=q_4+1=2$, $q_2=2q_3=4$ and
$q_1=2q_2=8$ which completes the proof of Theorem A. \qed


\begin{thebibliography}{10}

\bibitem{altman}
{\sc E.~Altman, B.~Gaujal and A.~Hordijk}, {\em Balanced sequences and optimal routing},
  Assoc. Comput. Mach., 47 (2000), pp.~752--775.


\bibitem{barat}
{\sc J.~Bar\'{a}t and P.Varj\'{u}}, {\em Partitioning the positive integers to seven Beatty sequences},
 Indag. Math., 14 (2003), pp.~149--161.



\bibitem{beatty}
{\sc S.~Beatty}, {\em Problems and Solutions},
  Amer. Math. Monthly, 34 (1927), pp.~159--160.


 \bibitem{BFF86_2}
{\sc M.A.~Berger, A.~Felzenbaum and A.S.~Fraenkel}, {\em Disjoint covering systems of rational Beatty sequences},
 J. Combin. Theory Ser. A, 42 (1986), pp.~150--153.

 \bibitem{EG} {\sc Erd\H{o}s, Paul, and Ronald L. Graham}, Old and new problems and results in combinatorial number theory. Vol. 28. L'Enseigenemet mathématique, 1980.

 \bibitem{fraenkel73}
{\sc A.S.~Fraenkel}, {\em Complementing and exactly covering sequences},
  J. Combinatorial Theory Ser. A, 14 (1973), pp.~8--20.


\bibitem{graham73}
{\sc R.L.~Graham}, {\em Covering the positive integers by disjoint sets of the form $\{\{[n\alpha+\beta]\}: N= 1, 2,...\}$},
  J. Combinatorial Theory Ser. A, 15 (1973), pp.~354--358.


\bibitem{poonen}
{\sc B.~Poonen}, {\em The number of intersection points made by the diagonals of a regular polygon},
  SIAM J. Discrete Math., 11 (1998), pp.~135--156.


\bibitem{morikawa}
{\sc R.~Morikawa}, {\em On eventually covering families generated by the bracket function},
  Bull. Fac. Liberal Arts Nagasaki Univ., 23 (1982), pp.~17--22.

\bibitem{strutt}
{\sc Rayleigh, John William Strutt}, {\em The theory of sound},
  Dover publications. New york, (1877).


\bibitem{simpson91}
{\sc J.~Simpson}, {\em Disjoint covering systems of rational Beatty sequences},
  Discrete Math., 92 (1991), pp.~361--369.

\bibitem{simpson2004}
{\sc J.~Simpson}, {\em Disjoint Beatty sequences},
  Integers, 4 (2004), A12.


\bibitem{tijdeman96}
{\sc R.~Tijdeman}, {\em On complementary triples of Sturmian bisequences},
  Indag. Math. (N.S.), 7 (1996), pp.~419--424.


\bibitem{tijdeman}
{\sc R.~Tijdeman}, {\em Fraenkel's conjecture for six sequences},
  Discrete Math., 222 (2000), pp.~223--234.

\bibitem{3gapth}
{\sc T.~van Ravenstein}, {\em The three gap theorem ({S}teinhaus conjecture).
 J. Austral. Math. Soc. Ser. A}, 45(1988), pp.~360--370.



















%
%
%
%
%
%
%
%
%
%
%
%

\end{thebibliography}
\end{document}